\documentclass[12pt,reqno]{amsart}
\usepackage{amsmath,amssymb}
\newcommand{\vectornorm}[1]{\left|\left|#1\right|\right|}
\newcommand{\R}{\mathbb{R}}

\def\ds{\displaystyle}

\newtheorem{theorem}{Theorem}[section]
\newtheorem{lemma}[theorem]{Lemma}

\newtheorem{proposition}[theorem]{Proposition}

\newtheorem{definition}[theorem]{Definition}

\numberwithin{equation}{section}

\begin{document}

\title[Saddle-shaped solutions]{Saddle-shaped solutions of bistable
diffusion equations in all of $\R^{2m}$}

\author{Xavier Cabr{\'e}}
\thanks{Both authors were supported by MTM2005-07660-C02-01.
This work was part of the ESF Programme ``GLOBAL''}
\address{ICREA and Universitat Polit{\`e}cnica de Catalunya,
Departament de Matem{\`a}-tica Aplicada I, Diagonal 647, 08028
Barcelona, Spain}
\email{xavier.cabre@upc.edu}

\author{Joana Terra}
\thanks{The second author was supported by the FCT grant
SFRH/BD/8985/2002}
\address{Universitat Polit{\`e}cnica de Catalunya,
Departament de Matem{\`a}tica  Aplicada I, Diagonal 647, 08028
Barcelona, Spain}
\email{joana.terra@upc.edu}

\begin{abstract}
We study the existence and instability properties of saddle-shaped 
solutions of the semilinear elliptic equation $-\Delta u = f(u)$ 
in the whole $\R^{2m}$, where $f$ is of bistable type. 
It is known that in dimension $2m=2$ there exists a saddle-shaped 
solution. This is a solution which changes sign in $\R^2$ and 
vanishes only on $\{|x_1|=|x_2|\}$.  
It is also known that this solution is unstable. 

In this article we prove the existence of saddle-shaped 
solutions in every even dimension, as well as their instability
in the case of dimension $2m=4$. 
More precisely, our main result 
establishes that if $2m=4$, every solution vanishing on the
Simons cone $\{(x^1,x^2)\in\R^m\times\R^m : |x^1|=|x^2|\}$
is unstable outside of every compact set and, as a consequence, 
has infinite Morse index. These results are relevant in connection with
a conjecture of De Giorgi extensively studied in recent years
and for which the existence of a counter-example in high dimensions
is still an open problem.
\end{abstract}

\maketitle

\section{Introduction}

This paper is concerned with the study of bounded solutions
of bistable diffusion equations
\begin{equation}\label{eq}
-\Delta u=f(u)\quad {\rm in }\,\R^{n}.
\end{equation}
In the last years there has been special interest
in a symmetry property of certain solutions. 
It consists of establishing whether every monotone solution
$u$ of $\eqref{eq}$ depends only on one Euclidean variable or, equivalently,
whether the level sets of such solutions are all hyperplanes. This  question was
raised by De Giorgi~\cite{DG2} in $1978$, who conjectured that the level sets of
every bounded, monotone in one direction, solution of the Allen-Cahn equation
\begin{equation}\label{GL}
-\Delta u=u-u^3\quad{\rm in }\,\R^n,
\end{equation}
must be hyperplanes, at least if $n\leq 8$. The conjecture has 
been proven to be
true when the dimension $n=2$ by Ghoussoub and Gui \cite{GG}, and when $n=3$ by
Ambrosio and Cabr\'e \cite{AC}. For $4\leq n\leq 8$ and assuming an additional limiting 
condition on~$u$, it has been established by Savin \cite{OS}
(see section $2$ for more details). 

It remains open
the existence of a counter-example in higher dimensions. By a result of
Jerison and Monneau \cite{JM}, the
existence of a counter-example to the conjecture in $\R^{n+1}$ would be
established if one could prove the
existence of a bounded, even with respect to each coordinate, global
minimizer of \eqref{GL} in~$\R^{n}$. By global minimizer we mean an absolute 
minimizer of the energy with respect to
compactly supported perturbations. On the other hand, by a deep result of
Savin \cite{OS}, for $n\leq 7$ every global minimizer is an odd
function of only one Euclidean variable. In particular, an even function
with respect to each coordinate can not be a global
minimizer in  $\R^{n}$ whenever $n \leq 7$. 

The crucial remaining question is whether a global minimizer of \eqref{GL}, 
even with respect to each coordinate, exists in higher dimensions. A natural
candidate is expected to be found in the class of saddle-shaped
solutions, that is, solutions that depend only on two radial
variables $s=|x^1|$ and $t=|x^2|$, change sign in
$\R^n=\R^{2m}=\{(x^1, x^2)\in\R^m\times\R^m\}$ and 
vanish only on the Simons cone ${\mathcal C}=\{s=t\}$. 
This cone is of importance in the
theory of minimal surfaces and its variational properties are related
to the conjecture of De Giorgi.  Namely, the cone 
${\mathcal  C}\subset\R^{2m}$ has zero mean curvature in all even dimensions
(except at the singular point $0$), but it 
is a minimal cone (minimal in the variational sense) 
if and only if $2m\geq 8$ (see section $2$).

Towards the understanding of this open question on global minimizers,
we study here saddle-shaped solutions and their stability properties.
To be precise in our statements, we first present the definitions 
to be used throughout the paper. 

Equation ($\ref{eq}$) is
the Euler-Lagrange equation associated to the energy functional
\begin{equation}\label{energia}
{\mathcal E}(v,\Omega):=\int_\Omega \left\{ \frac{1}{2}|\nabla
v|^2+G(v)\right\} dx,\qquad\text{where } \, G'=-f
\end{equation}
and $\Omega\subset\R^n$ is a bounded domain. The
energy ${\mathcal E}$ leads to the following
notions on minimality, stability, and Morse index of bounded
solutions.

\begin{definition}\label{def1.1}
{\rm Let $f\in C^1(\R)$.}

\renewcommand{\labelenumi}{$($\alph{enumi}$)$}
\begin{enumerate}
\item {\rm 
We say that a bounded $C^1$ function $u:\R^n\rightarrow\R$ 
is a {\it global minimizer} of $(\ref{eq})$ if
$${\mathcal E}(u,\Omega)\leq {\mathcal E}(u+\xi,\Omega),$$ for
every bounded domain $\Omega$ and every $C^{\infty}$ function
$\xi$ with compact support in $\Omega$.
}
\item {\rm 
We say that a bounded solution $u$ of
$(\ref{eq})$ is {\it stable} if the second variation of energy
$\delta^2{\mathcal E}/\delta^2\xi$ with respect to compactly
supported perturbations $\xi$ is nonnegative. That is, if
\begin{equation}\label{stable}
Q_u(\xi):=\int_{{\R}^n}\left\{
|\nabla\xi|^2-f'(u)\xi^2\right\}dx\geq 0 \quad {\rm for \, all}
\;\xi\in C^\infty_c(\R^n).
\end{equation}
We say that $u$ is {\it unstable} if and only if $u$ is not stable.
}
\item {\rm 
We say that a bounded solution $u$ of
$(\ref{eq})$
has {\it finite Morse index} equal to $k\in\{0,1,2,\ldots\}$ if $k$ is the 
maximal dimension of a subspace $X_k$ of $C^1_c (\R^n)$ such that 
$Q_u(\xi) < 0$ for every $\xi\in X_k \setminus \{0\}$. Here
$C^1_c (\R^n)$ is the space of $C^1(\R^n)$ functions with compact support 
and $Q_u$ is defined in \eqref{stable}. 
If there is no such finite integer $k$, we then say that $u$ has
{\it infinite Morse index}.
}
\end{enumerate}
\end{definition}

Clearly, every global minimizer is a stable solution. At the same time, 
every stable solution has finite Morse index equal to $0$. 
It is also easy to verify that every solution with finite Morse
index is stable outside of a compact set (see Theorem~\ref{uns}
and its proof for more details). In some
references, global minimizers are called ``local minimizers'', where
local stands for the fact that the energy is computed in bounded domains.

The following assumption on $G$,
\begin{equation}\label{condG}
G\geq 0=G(\pm M) \; {\rm in} \,\R \quad{\rm and}\quad G>0 \;{\rm in}\,
(-M,M)
\end{equation}
for some constant $M>0$, guarantees  the existence of an
increasing solution of $(\ref{eq})$ in dimension 1, that is 
in all of $\R$, taking values onto
$(-M,M)$; see Lemma $\ref{lemma1D}$. In addition,
such increasing solution is unique up to translations of the
independent variable. Normalizing it to vanish at the origin, we
call it~$u_0$. Thus, we have
\begin{equation}\label{u0}
\left\{
\begin{array}{l}
u_0:\R\rightarrow (-M,M) \\
u_0(0)=0,\, \dot{u}_0>0, \;{\rm and}\\
 -\ddot{u}_0=f(u_0) \quad {\rm in}\, \R.
 \end{array}
 \right.
\end{equation}
We will see that \eqref{condG} is related to the bistable character of $f$.
Hypothesis \eqref{condG} is satisfied by $f(u)=u-u^3$, for which
$G(u)=(1/4)(1-u^2)^2$ and $M=1$. For this nonlinearity, 
the solution $u_0$ can be computed explicitly and it is given by 
$u_0(\tau)=\tanh(\tau/\sqrt{2})$.

Next, note that for every given $b\in\R^n$ with $|b|=1$ and $c\in\R$, the function
\begin{equation}\label{1d}
u_{b,c}(x)=u_0(b\cdot x+c)\quad \textrm{ for }  x\in\R^n,
\end{equation}
is a bounded solution of $(\ref{eq})$. These solutions are called
$1$-D solutions since they depend only on one Euclidean variable. Equivalently,
these are the solutions with every of their
level sets being a hyperplane. As a consequence of a
result of Alberti, Ambrosio, and the first author \cite{AAC}, it is
known now that, under hypothesis ($\ref{condG}$) on the nonlinearity, 
every $1$-D solution $u_{b,c}$ is a
global minimizer of $(\ref{eq})$. In particular, $u_{b,c}$ is a stable solution.

Furthermore, by a
result of Savin~\cite{OS}
in connection with the conjecture of De Giorgi, we know
now that $1$-D solutions are the only global minimizers of the Allen-Cahn 
equation~\eqref{GL}
in $\R^n$ for $n\leq 7$. On the other hand, as mentioned before 
(see Theorem~\ref{jermon} in next section
for more details), trying to find 
a counter-example to the conjecture in higher dimensions 
(still an open problem) is related to the possibility of finding certain 
global minimizers in dimensions $n\geq 8$ which are not $1$-D. 
More precisely, the
existence of a counter-example to the conjecture in $\R^{n+1}$ would be
established if one could prove the
existence of a bounded, even with respect to each coordinate, global
minimizer of \eqref{GL} in $\R^{n}$.
Natural candidates to be minimizers of this type are certain saddle-shaped solutions.
The study of their existence and stability properties is the goal
of this paper.

The type of saddle-shaped solutions that we consider are expected to have
relevant variational properties due to
a well known connection between 
semilinear equations modelling phase transitions 
and the theory of minimal surfaces (see section
$2$ for details). Such connection also motivated De
Giorgi to state his conjecture. More precisely, the saddle 
solutions that we consider are odd with respect
to the Simons cone, which is defined as follows. For $n=2m$, the Simons cone
is given by
\begin{equation}\label{Sim}
{\mathcal C} = \{x\in\R^{2m}: x_1^2 + x_2^2 + \cdots +
x_m^2=x_{m+1}^2 + x_{m+2}^2 + \cdots + x_{2m}^2\}.
\end{equation}
It is easy to
verify that ${\mathcal C}$ has zero mean curvature at every
$x\in{\mathcal C}\backslash\{0\}$, in every dimension $2m\geq 2$.
However, it is only in dimensions $2m\geq 8$ (besides the case $2m=2$)
that this cone is
locally stable. In dimensions $2m\geq 8$ it is in addition a
minimizer of the area functional, that is, it is a minimal cone 
(in the variational sense); see~\cite{G}.

For $x=(x_1,\dots, x_{2m})\in\R^{2m}$, we define two
radial variables $s$ and $t$ by
\begin{equation}\label{coor}
\left\{\begin{array}{rcll} s& = & {\ds \sqrt{x_1^2+...+x_m^2}}& \geq 0\\
 t & = &{\ds \sqrt{x_{m+1}^2+...+x_{2m}^2}}& \geq 0.
\end{array}
\right.
\end{equation}
The Simons cone is given by ${\mathcal C}=\{s=t\}$.

We now introduce our notion of saddle solution. These
solutions depend only on $s$ and $t$, and are odd with
respect to $\mathcal C$.

\begin{definition}\label{def} 
{\rm Let $f\in C^1(\R)$ be odd. We  say  that 
$u:\R^{2m}\rightarrow\R$ is a} {\it saddle-shaped
solution} {\rm  (or simply a saddle solution) of 
\begin{equation}\label{eq2m}
-\Delta u=f(u) \quad {\rm in }\;\R^{2m}
\end{equation}  if $u$ is a bounded solution of
$(\ref{eq2m})$ and, with $s$ and $t$ defined by \eqref{coor},}
\renewcommand{\labelenumi}{$($\alph{enumi}$)$}
\begin{enumerate}
\item{\rm  $u$ depends only on the variables $s$ and $t$. We write
$u=u(s,t)$; \item $u>0$ in ${\mathcal O}=\{s>t\}$; \item
  $u(s,t)=-u(t,s)$ in $\R^{2m}$.}
\end{enumerate}
\end{definition}

It follows from (c) that every saddle solution vanishes on the 
Simons cone ${\mathcal C}=\{s=t\}$. Note also that saddle solutions are
even with respect to each coordinate $x_i$, $1\le i\le 2m$, as
in the result of Jerison-Monneau.

By classical elliptic regularity theory, it is well known that
for $f\in C^1(\R)$, every bounded solution of $-\Delta u=f(u)$ in $\R^{n}$
satisfies $u\in C^{2,\alpha} (\R^n)$ for all $0<\alpha <1$, and
thus it is a classical solution. In particular, saddle solutions
are classical solutions. See the beginning of section~3
for more details.

In Theorem~\ref{exis} below, we establish the existence of 
a saddle solution in every even dimension for odd bistable 
nonlinearities. This will be accomplished using odd
reflection with respect to the cone ${\mathcal C}$, after
constructing with a variational technique 
a positive solution in
${\mathcal O}=\{s>t\}$ depending only on $s$ and $t$.

Saddle solutions were first studied by Dang, Fife, and 
Peletier~\cite{DFP}  in dimension $n=2$ for $f$ odd, bistable, 
and with $f(u)/u$ decreasing for $u\in(0,1)$. They proved the
existence and uniqueness of a saddle solution in dimension $2$. 
They also established monotonicity properties
and the asymptotic behavior of the saddle solution. 
Its instability, already indicated in a partial
result of~\cite{DFP}, was studied in detail by
Schatzman~\cite{Sc}. This paper established that the saddle solution is 
unstable in $\R^2$ by studying the linearized 
operator at the solution in some appropriate functional spaces,
and by showing that  it has a strictly negative eigenvalue corresponding to an
eigenfunction having the symmetries of the square. Moreover, in
the case of the Allen-Cahn equation $(\ref{GL})$, the
linearized operator was shown to have exactly one negative eigenvalue. 

The article \cite{ABG} studies vector-valued saddle solutions in $\R^2$.
The recent work~\cite{ACM} concerns scalar saddle type solutions in
$\R^2$ changing sign on more nodal lines than $x_1=\pm x_2$.

The instability of the saddle solution in dimension $2$ 
(in the sense of Definition~\ref{def1.1}) is nowadays a
consequence of a more recent result related to the conjecture of De
Giorgi. Namely, \cite{GG} and \cite{BCN} 
established that, for all $f\in C^1$, every bounded stable solution 
of ($\ref{eq}$) in $\R^2$  must be a $1$-D solution, that is, a solution
depending only on one Euclidean variable. 
In particular, in $\R^2$ bounded stable solutions can not be 
saddle-shaped.
These ideas were further used in~\cite{Sh} when the dimension $n=2$.

To state our results on saddle solutions, given a $C^1$ nonlinearity  
$f:\R\rightarrow\R$ and $M>0$, define 
\begin{equation}\label{defG}
G(u)=\int_u^M f.
\end{equation}
We have that
$G\in C^2(\R)$ and
$G'=-f.$ In our results we assume some, or all, of the following
conditions on $f$. For some $M>0$, and with $G$ defined as above,
consider the following properties of $f$:
\begin{eqnarray}
& f \text{ is odd;} \label{H1}\\
& G\geq 0=G(\pm M) \text{ in } \R \text{ and } G>0 \text{ in } (-M,M); \label{H2}\\
& f \text{ is concave in } (0,M).\label{H3}
\end{eqnarray}

Condition \eqref{H2} is actually condition ($\ref{condG}$) presented before
in relation with the existence of $1$-D solutions.
Note that if \eqref{H1} and \eqref{H2} hold, then
$f(0)=f(\pm M)=0$. On the other hand, if $f$ is odd in~$\R$, positive and
concave in $(0,M)$, and negative in $(M,\infty)$, then $f$ satisfies
\eqref{H1}, \eqref{H2}, and \eqref{H3}.  Hence, the nonlinearities $f$ that we
consider are of ``balanced bistable type", while the potentials $G$
are of ``double-well type". Our three assumptions \eqref{H1}, \eqref{H2}, \eqref{H3} are
satisfied for the Allen-Cahn (or scalar Ginzburg-Landau)
equation $-\Delta u= u-u^3.$
In this case we have that $G(u)=(1/4)(1-u^2)^2$ and
 $M=1$. The three hypothesis also hold
for the equation $-\Delta u=\sin (\pi u)$, for which
$G(u)=(1/\pi)(1+\cos (\pi u))$.

Our first result establishes the existence of a saddle solution  in~$\R^{2m}$
and some of its variational properties.

\begin{theorem}\label{exis} 
Let $f\in C^1(\R)$ satisfy \eqref{H1} and \eqref{H2} for some constant $M>0$,
where $G$ is defined by \eqref{defG}. Then, for every even dimension $2m\geq 2$,
there exists a saddle-shaped solution $u$ as in Definition~{\rm \ref{def}} 
of $-\Delta u= f(u)$ in~$\R^{2m}$.

In addition, $u$ satisfies $|u|<M$ in $\R^{2m}$, as well as the energy estimate
\begin{equation}\label{est-ener}
{\mathcal E}(u,B_R)=\int_{B_R} \left\{ \frac{1}{2}|\nabla
u|^2+G(u)\right\} dx \leq CR^{2m-1} \qquad\text{for all }R>1,
\end{equation}
where $C$ is a constant independent of $R$ and $B_R$ denotes the
open ball of radius~$R$ centered at $0$.

If in addition $f$ satisfies \eqref{H3}, then
the second variation of energy $Q_u(\xi)$
at~$u$, as defined in \eqref{stable}, is nonnegative for all
functions $\xi\in C^1(\R^{2m})$ with compact support
in $\R^{2m}$ and vanishing on the
Simons cone $\mathcal{C}=\{s=t\}$.
\end{theorem}

As a consequence of the last statement in the theorem,
the instability
of saddle solutions in low dimensions is related to
perturbations which do not vanish on the
Simons cone, and hence, which change the zero level set of the solution.

We prove the existence of 
a saddle solution by first constructing
a positive solution in
${\mathcal O}=\{s>t\}$ depending only on $s$ and $t$. 
For this, we use a variational method.
We then obtain the saddle solution in all 
the space through odd
reflection with respect to the cone ${\mathcal C}$.

Further variational and monotonicity properties of saddle solutions, 
as well as their asymptotic behavior, will be 
established in a forthcoming article \cite{CT} by the same
authors.

Note that for functions $u$ depending only on $s$ and $t$,
such as saddle solutions, the energy functional \eqref{energia} becomes
\begin{equation}\label{enerst}
{\mathcal E}(u,\Omega)=a_m \int_\Omega s^{m-1}t^{m-1}\left\{ \frac{1}{2}
(u_s^2+u_t^2)+G(u)\right\} ds dt,
\end{equation}
where $a_m$ is a positive constant depending only on $m$
---here we have assumed that $\Omega\subset\R^{2m}$ is radially symmetric
in the first $m$ mariables and also in the last $m$ variables,
and we have abused notation by identifying $\Omega$ with its
projection in the $(s,t)$ variables. 
In these variables, the semilinear equation \eqref{eq2m} reads
\begin{equation}\label{eqst}
-(u_{ss}+u_{tt})-(m-1){\Big (}\frac{u_s}{s}+\frac{u_t}{t}{\Big
)}=f(u)\qquad \text{for } s>0,\, t>0. 
\end{equation}

The following is our main result. In dimension $n=4$, we establish 
the instability outside of every compact set
of all bounded solutions (not necessarily depending on $s$ and $t$ only)
that vanish on the Simons cone ${\mathcal C}=\{s=t\}$. As a consequence, 
the Morse index of such solutions is proved to be infinite.

\begin{theorem}\label{uns} Let $f\in C^1(\R)$ satisfy 
\eqref{H1}, \eqref{H2}, and \eqref{H3}.
Then, every bounded solution of $-\Delta u=f(u)$ in $\R^4$ that
vanishes on the Simons cone 
${\mathcal C}=\{x_1^2+x_2^2=x_3^2+x_4^2\}$ is unstable. 
Furthermore, every such solution $u$ is unstable outside of every
compact set. That is, for every compact set $K$ of $\R^4$
there exists $\xi\in C^1(\R^4)$ with compact support in
$\R^4\setminus K$ for which 
$Q_u(\xi) < 0$, where $Q_u$ is defined in \eqref{stable}. 
As a consequence, $u$ has infinite Morse index in the sense of
Definition~$\ref{def1.1}$.

In particular, all the previous statements hold true for
every saddle-shaped solution as in Definition~$\ref{def}$
if $2m=4$.
\end{theorem}

As mentioned before, the instability of the saddle solution 
in dimension~$2$ was already proven by Shatzman \cite{Sc}.  
More recently we have established the instability result also
in dimension $6$ ---this is to appear in a
forthcoming paper~\cite{CT}. The computations in section~6 of the
present paper, and the more delicate ones in \cite{CT}, suggest the
possibility of saddle solutions being stable in dimensions
$2m\geq 8$. Such stability result would be a promising hint towards
the possible global minimality of saddle solutions in 
high dimensions, and hence towards a counter-example to 
the conjecture of De Giorgi.

A crucial ingredient in the proof of Theorem
$\ref{uns}$ is the following pointwise estimate.

\begin{proposition}\label{prop} 
Let $f\in C^1(\R)$ satisfy \eqref{H1} and \eqref{H2}. 
If $u$ is a bounded solution of $-\Delta u=f(u)$ in $\R^{2m}$ that
vanishes on the Simons cone ${\mathcal C}=\{s=t\}$, then
\begin{equation}\label{estimate}
|u(x)|\leq |u_0({\rm dist}(x,{\mathcal C}))|=\left| u_0{\Big
(}\frac{s-t}{\sqrt{2}}{\Big )}\right| \quad \text{ for all }
x\in\R^{2m},
\end{equation}
where $u_0$ is defined by $(\ref{u0})$ and 
${\rm dist}(\cdot,{\mathcal C})$ denotes the distance 
to the Simons cone.

In addition, the function $u_0((s-t)/\sqrt{2})$ is a 
supersolution of $-\Delta u=f(u)$ in the set
${\mathcal O}=\{s>t\}$.

\end{proposition}
This proposition is proven in section $4$ using an important 
gradient bound of Modica \cite{M1} for bounded solutions of
($\ref{eq}$). Instead, its last statement 
---~$u_0((s-t)/\sqrt{2})$ being a supersolution in 
$\{s>t\}$, which by the way we will not use in this paper--- 
follows simply from direct computation using
\eqref{eqst}. Since  $|s-t|/\sqrt{2}$ is the distance to the
Simons cone, this last statement corresponds to the well known
fact that the distance function to a hypersurface of zero
mean curvature is superharmonic in each side of the
hypersurface.

The heuristic idea behind the instability result of
Theorem~\ref{uns} is the following. One expects that the 
saddle solution behaves at infinity as the transition
profile $u_0$ placed over the cone $\mathcal C$, that is, as 
$u_0((s-t)/\sqrt 2)$ in \eqref{estimate}. One may expect that 
this, 
combined with the instability of the Simons cone in dimensions 4 and 6,
could lead to the instability of the saddle
solution.  In this paper we see that this idea works in 
dimension 4 thanks to the estimate of Proposition~\ref{prop}.

Indeed, the proof of Theorem~\ref{uns} proceeds as follows.
We prove that the quadratic form $Q$ defined by
($\ref{stable})$ with the solution $u$ replaced by 
the explicit function $u_0 ((s-t)/\sqrt 2)$, 
is negative when $n=4$ for some test function $\xi$. This will imply 
---based on estimate~\eqref{estimate} and the
assumptions on $f$--- that $Q_u$ is also negative for some 
test function, where $u$ is any given solution vanishing on
${\mathcal C}$. That is, $u$ is unstable in dimension $n=4$.

Finally, let us make a comment on results about the Morse index 
of stationary surfaces, i.e., surfaces of zero mean curvature. 
The usual proof of 
the instability of the Simons cone in dimension 4 and 6 (see \cite{G}) 
also leads to its instability outside of every compact set, and hence
to the infinite Morse index property. A precise study of the Morse index of 
stationary surfaces close to the Simons cone is made in~\cite{A2}
through the analysis of intersection numbers.

The paper is organized as follows. In section $2$ we present the
precise statement of the conjecture of De Giorgi and its connections
with the variational properties of solutions to ($\ref{GL}$) and 
with minimal cones. We also recall the result of 
Jerison and Monneau mentioned above. Section $3$ is
devoted to the proof of Theorem $\ref{exis}$ on the existence
of saddle solution. Section $4$ concerns the proof of Proposition
$\ref{prop}$, an important tool towards the proof of our
instability result, Theorem~$\ref{uns}$, which is presented in
section $5$. Finally, in section~$6$ we 
present the asymptotic computations used in the proof of
Theorem~$\ref{uns}$ carried out in every dimension $2m\geq 4$.

\section{A conjecture of De
Giorgi, minimal cones, and saddle solutions}

In 1978 De Giorgi \cite{DG2} raised the following question:

\noindent {\bf Conjecture (De Giorgi \cite{DG2})} 
{\it Let $u\in C^2(\R^n)$ be a solution of
$$
-\Delta u = u-u^3 \qquad\hbox{in }\R^n
$$
such that
$$
\vert u\vert\leq 1\ \quad\hbox{and}\quad \partial_{x_n} u>0
$$ in the whole $\R^n$. Is it true that all level sets
$\{u=\lambda\}$ of $u$ are hyperplanes, at least if $n\le 8\,${\rm
?} }

This conjecture was proved for $n=2$ by Ghoussoub and
Gui~\cite{GG}, and for $n=3$ by Ambrosio and Cabr{\'e}~\cite{AC}. 
For $n \leq 8$, a weaker
version of the conjecture was  proven recently by Savin~\cite{OS}.
Namely, if one further assumes that 
\begin{equation}\label{lim}
 \lim_{x_n\rightarrow\pm \infty}u(x',x_n)=\pm 1\qquad 
\text{for all }x'\in \R^{n-1}
\end{equation}
and $n\leq 8$, then all level sets of $u$ are hyperplanes. 
We emphasize that, in
this result, the limits above are not assumed to be uniform in
$x'\in \R^{n-1}$. 

A related and deep result of Savin \cite{OS} is the following. If
$u$ is a global minimizer of \eqref{GL}
(a local minimizer in the terminology of
\cite{OS}) and $n\leq 7$, then the level sets of $u$ are hyperplanes. 
One expects that $n\leq 7$ is optimal in this result. However,
the existence for $n\geq 8$ of a global minimizer not being 1-D 
is still an open problem. In this direction, 
saddle solutions are natural candidates for
being global minimizers (and not 1-D) in high dimensions. More
precisely, if their minimality hold true in
some dimension, this 
would provide a counter-example to the
conjecture of De Giorgi in one more dimension.
 
Indeed, the connection between the existence of 
certain global minimizers and the
veracity of the conjecture of De Giorgi is established by Jerison
and Monneau in \cite{JM}. Namely, they
prove that if there exists a bounded, even with respect to
each coordinate, global minimizer
in $\R^{n-1}$, then there would be a bounded 
solution $u$ to $(\ref{GL})$, increasing in $x_n$ and with one level set
not being a hyperplane. That is, this would provide a
counter-example to the conjecture in $\R^n$. 
Their precise result is the following.

\begin{theorem}{\bf (Jerison-Monneau \cite{JM})}
Let $G$ satisfy \eqref{H2} with $M=1$ and assume that there
exists a global minimizer $v$ in $\R^{n-1}$ such that $|v|<1$ and
$v$~is even with respect to each coordinate $x_i$, $i=1,\ldots,n-1$. 

Then, for each $\gamma\in(0,\sqrt{2G(v(0))})$,  there exists a
solution $u\in C^2(\R^n)$ of $\Delta u=G'(u)$ in $\R^n$
satisfying $|u|\leq 1$
and $\partial_{x_n}u>0$ in $\R^n$, and such that, for one
$\lambda\in\R$, the set $\{u=\lambda\}$ is not a hyperplane.

Moreover, this solution $u$ is a global minimizer in $\R^n$, it is
even in the first $n-1$ coordinates, and satisfies 
$\partial_{x_n}u(0)=\gamma$ and $u(0)=v(0)$.
\label{jermon}
\end{theorem}

It is not known, however, if the solution $u$ of the previous
theorem would satisfy $\lim_{x_n\to\pm\infty}u(x',x_n)=\pm 1$ for all
$x'\in\R^n$ as in hypothesis \eqref{lim}.

Note that saddle solutions (as in Definition~\ref{def}) are
even with respect to each coordinate $x_i$, $1\le i\le 2m$,
as in the previous theorem (here we would have $n-1=2m$).

As a first step towards understanding global minimizers with the
properties of~$v$ in Theorem~\ref{jermon}, 
one may study stable solutions ---stability being a necessary 
condition for global minimality.
Classifying all bounded stable solutions to
($\ref{eq}$) is a difficult task. A complete
characterization of stable solutions is only available 
in dimension $n=2$. The
results of \cite{GG} imply that, for all $f$, a nonconstant bounded solution
to ($\ref{eq}$) in $\R^2$ is stable if and only if it is 1-D and
monotone. The proof of this stability result involves a
Liouville-type theorem due to Berestycki, Caffarelli, 
and Nirenberg~\cite{BCN}.
As an immediate corollary of the previous result, we get
that, for $n=2$, if $u$ is radially symmetric then it is unstable, since all
stable solutions must be 1-D. In higher dimensions, Cabr{\'e} and
Capella \cite{CC} have established that, for $n\leq 8$ and all 
$f\in C^1(\R)$, if $u$ is a nonconstant bounded radial solution of
($\ref{eq}$), then $u$ is unstable. The same result holds for
$9\leq n\leq 10$ by a more recent result of Villegas~\cite{Vi}.
On the other hand, for $n\geq 11$, \cite{CC} constructs a polynomial 
$f$ which admits a stable nonconstant bounded radial solution 
$u$ of ($\ref{eq}$). Recent works of Dancer and of Farina \cite{F1,F2,DF}
establish interesting classification results for 
stable and finite Morse index solutions (general solutions,
not only radial, and even unbounded) of supercritical elliptic problems.

The level sets of 1-D solutions and of radial solutions are, 
respectively, hyperplanes and spheres. Instead, 
saddle solutions have the Simons cone as one level set, and
thus their geometry is more involved. In what remains
of this section we explain the results 
on minimal graphs and minimal cones that are relevant in
connection with the conjecture of De Giorgi and with the variational
properties of saddle solutions.

Let $u$ be a bounded solution of ($\ref{GL}$) in all of
$\R^n$ and consider the blow-down family of functions
$\{u_{\varepsilon}\}$ defined by
$u_{\varepsilon}(x)=u(x/\varepsilon)$, for small $\varepsilon$.
This is a solution of the same equation with $f$ replaced by
$\varepsilon^{-2}f$. The study of the behavior of $u_{\varepsilon}$ as
$\varepsilon\rightarrow 0$ leads to some information on
$u$ at infinity.
It was proven by Modica and Mortola \cite{MM1} that
the energy functionals ${\mathcal
E}_{\varepsilon}$ corresponding to $u_{\varepsilon}$
(see \cite{AAC} for details) $\Gamma$-converge to
a multiple of the perimeter functional $\mathcal P$ as
$\varepsilon\rightarrow 0$. Let us explain this with one of its
consequences. If $\{u_{\varepsilon}\}$ is a
sequence of minimizers of ${\mathcal E}_{\varepsilon}$, 
then a subsequence of $u_{\varepsilon}$
converges to a characteristic function $\ds{\chi_{E} -
\chi_{\Omega\setminus E}}$ as $\varepsilon\rightarrow 0$ for which
$\partial E \cap \Omega$ is a minimal hypersurface (minimal in the
variational sense).

Since the level
sets $\{u_{\varepsilon}=\lambda\}$ are rescaled versions of the level sets
$\{u=\lambda\}$ of~$u$, the result of
Modica and Mortola indicates that the level sets
$\{u=\lambda\}$ of~$u$ converge at infinity, 
in some weak sense and after subsequences, to a minimal surface. 
The minimality of $u$, under the hypothese of the conjecture and \eqref{lim},
is guaranteed by a result of \cite{AAC}. Since~$u$ satisfies the monotonicity condition
\begin{equation*}
\partial_{x_n}u >0 \qquad {\rm in} \;\R^n,
\end{equation*}
then each level set of $u$ is the graph of a function from $\R^{n-1}$
to $\R$ along the $x_n$ direction. Therefore, the limiting minimal surface should
be the graph of a function from $\R^{n-1}$ to $\R$.

The problem of classifying all entire minimal graphs was
settled by Simons in \cite{S}. His result establishes that every entire
minimal graph of a function from $\R^k$ to $\R$ is necessarily a hyperplane
for $k\leq 7$. Going back to our problem, we should have that the
limiting minimal graph is a hyperplane (that is, the level sets
of $u$ are in some sense flat at infinity) whenever $k=n-1\leq
7$, i.e., when $n\leq 8$. The conjecture of De Giorgi raises the
question of whether or not each level set itself is a hyperplane,
and not only their limit at infinity and after subsequences.

To prove Simons result on minimal graphs, one first studies
minimal cones; see \cite{G}. Simons \cite{S} 
proved that all  minimal cones of dimension
less or equal than 6 living in $\R^n$ for $n\leq 7$ are
hyperplanes. In addition, he established the existence of 
a singular cone of dimension ${2m-1}$ living in $\R^{2m}$
with zero mean curvature (except at its vertex) and being locally
stable for the area functional if
$2m\geq 8$. This cone, known as Simons cone, is defined by
$$
{\mathcal C} = \{x\in\R^{2m}: x_1^2 + x_2^2 + \cdots + x_m^2=x_{m+1}^2 
+ x_{m+2}^2 + \cdots + x_{2m}^2\}.
$$
One year later, Bombieri, De Giorgi, and Giusti \cite{BGG}
proved that this cone is not only locally stable but actually a 
minimal cone, that is, a minimizer of the area
functional when $2m\geq 8$. Moreover, they proved that there
exists a minimal graph of a smooth function from $\R^k$ to $\R$
which is not a hyperplane when $k\geq 8$.

By our definition, the zero level set of a saddle solution 
coincides with the Simons cone. Hence we expect the
minimality properties of $\mathcal C$ in high dimensions to
play an important role in
the variational properties of saddle solutions.

\section{Existence of saddle solution in $\R^{2m}$}

In this section we prove the existence of a saddle solution
in every even dimension. 
Before this, let us recall some well known facts about the 
regularity of weak solutions. 

Every bounded solution of $-\Delta u=f(u)$ in $\R^{n}$, with $f\in C^1$,
satisfies $u\in C^{2,\alpha} (\R^n)$ for all $0<\alpha <1$. In addition, 
$\vert \nabla u\vert \in L^\infty(\R^n)$.
Indeed, we apply interior $W^{2,p}$ estimates, 
with $p>n$, to the equation 
in every ball $B_2(x)$ of radius $2$ in $\R^n$. We find that
\begin{equation}\label{rad1}
\begin{array}{ccl}
\Vert u \Vert_{C^1(\overline{B}_1(x))} &\leq &
C \Vert u \Vert_{W^{2,p}(B_1(x))}  \\
&\leq & C\left\{ \Vert u \Vert_{L^\infty(B_2(x))} +
\Vert f(u) \Vert_{L^p(B_2(x))}
\right\} \le C
\end{array}
\end{equation}
for some constant $C$ independent of $x\in\R^n$.
Next, we apply $W^{2,p}$ interior estimates to the equations
$-\Delta \, \partial_j u = f'(u)\, \partial_j u$, to obtain
$W^{3,p}$ and hence $C^{2,\alpha}$ estimates for $u$.

To prove the existence of a saddle solution in 
$\R^{2m}=\{x=(x^1,x^2)\in\R^m\times\R^m\}$, we consider the open set
$$
{\mathcal O}:=\{s>t \}=\{|x^1|>|x^2|\}\subset\R^{2m};
$$ 
note that
$$
\partial{\mathcal O}={\mathcal C}.
$$
Using a variational technique we will construct a solution $u$ 
in ${\mathcal O}$ satisfying $u>0$ in ${\mathcal O}$ 
and $u=0$ on ${\mathcal C}=\partial{\mathcal O}$.
Then, since $f$ is odd, by odd
reflection with respect to the cone ${\mathcal C}$ we obtain a 
saddle solution in the whole space. 

Let $B_R$ be the open ball
in $\R^{2m}$ centered at the origin and of radius $R$.
In the proof we will consider the open bounded set
$$
{\mathcal O}_R:={\mathcal O}\cap B_R=\{s>t \text{ and }
 |x|^2=s^2+t^2<R^2\}.
$$
Note that
$$
\partial {\mathcal O}_R=({\mathcal C}\cap \overline{B}_R)\cup
(\partial{B_R}\cap {\mathcal O}).
$$

Even that  in the proof we will not need the following fact, let us point out
that the sets ${\mathcal O_R}$ and ${\mathcal O}$ are domains 
(i.e., open connected sets) in dimensions $2m\geq 4$
(but clearly not in dimension $2$).
Indeed, to see that ${\mathcal O_R}$ is connected, let 
$x=(x^1,x^2)\in {\mathcal O}_R$.
We can arc-connect $x$ within ${\mathcal O}_R$ to the point
$(x^1,0)$, simply using the path $(x^1,\sigma x^2), 0\leq\sigma
\leq 1$. Finally, since $m\geq 2$, the point $(x^1,0)$ can be
arc-connected within $\{p\in\R^m : 0<|p|<R\} \times \{0\}$ 
to the point $(R/2,0,0,\ldots,0)$. Thus, ${\mathcal O}_R$ 
is arc-connected. It follows that ${\mathcal O}$
is also arc-connected.

\begin{proof}[Proof of Theorem $\ref{exis}$] 
With ${\mathcal O}_R$ defined as above, consider the space
$$
\tilde{H}_0^1({\mathcal O}_R)=\{v\in H_0^1({\mathcal O}_R) : 
v=v(s,t) \text{ a.e.}\}
$$ 
of $H_0^1$ functions in the bounded open set ${\mathcal O}_R$ which 
depend only on $s$ and~$t$. 
Equivalently, these are the $H_0^1({\mathcal O}_R)$
functions which are invariant under orthogonal transformations in the 
first $m$ variables and also under orthogonal transformations in the 
second $m$ variables. Thus, $\tilde{H}_0^1({\mathcal O}_R)$
is a weakly closed subspace of $H_0^1({\mathcal O}_R)$. 

Consider the energy functional in ${\mathcal O}_R$,
$$
{\mathcal E}(v,{\mathcal O}_R)=\int_{{\mathcal O}_R}
\left\{\frac{1}{2}|\nabla v|^2+ G(v)\right\}dx
\qquad \text{for } v\in \tilde{H}_0^1({\mathcal O}_R),
$$ 
defined on functions in $\tilde{H}_0^1({\mathcal O}_R)$.
Next we show the existence of a minimizer of the functional
among functions in this space.
Recall that we assume condition \eqref{H2} on $G$, that is,
$$
G\geq 0=G(\pm M) \text{ in } \R \text{ and } G>0 \text{ in } (-M,M).
$$

Since ${\mathcal E}$ is nonnegative, we can take a minimizing 
sequence $\{u_R^k\}$, $k=1,2,\ldots$, of ${\mathcal E}$ in 
$\tilde{H}_0^1({\mathcal O}_R)$. Without loss of generality we
may assume that $0 \leq u_R^k\leq M$. To see this, simply replace the 
minimizing sequence $\{u_R^k\}$ by the sequence 
$\{v_R^k\}$, defined by $v_R^k=\min\{|u_R^k|,M\}\in
\tilde{H}_0^1({\mathcal O}_R)$, which is also a minimizing 
sequence.  Indeed, $\{|u_R^k|\}$ is
a minimizing sequence since $G$ is even; then use that 
$G\geq G(M)$ to conclude that $\{v_R^k\}$ is also minimizing. 

Since $G\geq 0$, we have
$$
\int_{{\mathcal O}_R} |\nabla u_R^k|^2 \;dx \leq 
2{\mathcal E}(u_R^k,{\mathcal O}_R)\leq C
$$ 
for some constant C, and thus there exists a subsequence (denoted again by
$\{u_R^k\}$) such that $u_R^k$ converges weakly in
$H_0^1({\mathcal O}_R)$ to a function 
$u_R\in \tilde{H}_0^1({\mathcal O}_R)$.
Due to the weak convergence we deduce 
$$
\int_{{\mathcal O}_R} |\nabla u_R|^2  \;dx \leq 
\liminf_k\int_{{\mathcal O}_R} |\nabla u_R^k|^2 \;dx .
$$
By Fatou's lemma, we also have that
$$
\int_{{\mathcal O}_R} G(u_R) \;dx \leq \liminf_k 
\int_{{\mathcal O}_R} G(u_R^k) \;dx .
$$
Hence, $u_R$ is a minimizing function in $\tilde{H}_0^1({\mathcal O}_R)$
and $0\leq u_R\leq M$ in ${{\mathcal O}_R}$.

Next, we can consider perturbations $u_R+\xi$ of $u_R$, with
$\xi$ depending only on $s$ and $t$, and with $\xi$ having compact support in
${\mathcal O}_R\cap \{t>0\}=B_R\cap \{0<t<s\}$. In particular $\xi$ 
vanishes in a neighborhood
of $\{t=0\}$. Since equation $(\ref{eqst})$ in the $(s,t)$ variables
is the first variation of
${\mathcal E}(\cdot,{\mathcal O}_R)$ ---recall that ${\mathcal E}$
has the form \eqref{enerst} on $\tilde{H}_0^1$ functions---
and the equation is not singular 
away from $\{s=0\}$ and $\{t=0\}$, we deduce that $u_R$ is a 
solution of $(\ref{eqst})$ in ${\mathcal O}_R\cap \{t>0\}$. 
That is, we have
\begin{equation}\label{eq:t0}
-\Delta u_R=f(u_R)  \quad {\rm in }\; {\mathcal O}_R\cap \{t>0\}.
\end{equation}

We now prove that $u_R$ is also a solution in
all of ${\mathcal O_R}$, that is, also across  
$\{t=0\}$. To see this for dimensions $2m\geq 4$, let 
$\xi_{\varepsilon}$ be a smooth function of $t$ alone being identically $0$
in $\{t<\varepsilon/2\}$ and identically $1$
in $\{t>\varepsilon\}$. Let $v\in C^{\infty}_c({\mathcal O}_R)$,
multiply \eqref{eq:t0} by $v\xi_{\varepsilon}$ and integrate by parts to
obtain
\begin{equation}\label{capa}
\int_{{\mathcal O}_R} \xi_{\varepsilon}\nabla u_R\nabla v  \;dx
+ \int_{{\mathcal O}_R\cap \{t<\varepsilon\}} 
v \nabla u_R\nabla\xi_{\varepsilon}  \;dx
=\int_{{\mathcal O}_R}f(u_R)v\xi_{\varepsilon}  \;dx.
\end{equation}
We conclude by seeing that
the second integral on the
left hand side goes to zero as $\varepsilon\to 0$.
Indeed, by Cauchy-Schwartz inequality,
$$
\left| \int_{{\mathcal O}_R\cap \{t<\varepsilon\}} v\nabla u_R\nabla\xi_{\varepsilon} 
 \;dx \right|^2
\leq C \int_{{\mathcal O}_R\cap \{t<\varepsilon\}} |\nabla u_R|^2  \;dx 
\int_{{\mathcal O}_R\cap \{t<\varepsilon\}} |\nabla\xi_{\varepsilon}|^2 \;dx.
$$
Since $|\nabla\xi_{\varepsilon}|^2 \leq C/\varepsilon^2$,
$|{\mathcal O}_R\cap \{t<\varepsilon\}|\leq C_R \varepsilon^m$,
and $m\geq 2$, the second factor in the previous bound
is bounded independently of $\varepsilon$. At the same time,
the first factor tends to zero as $\varepsilon\to 0$
since $|\nabla u_R|^2$ is integrable in ${\mathcal O}_R$. 

In dimension $2m=2$ the previous proof does not apply and we
argue as follows. We now consider perturbations
$\xi\in \tilde{H}_0^1({\mathcal O}_R)$ which do not vanish on 
$B_R\cap \{t=0\}$. Considering the first variation of
energy and integrating by parts, we find that the 
boundary flux $s^{m-1}t^{m-1}\partial_t u_R=\partial_t u_R$ (here $m-1=0$) 
must be identically 0
on $B_R\cap \{t=0\}$. This implies that
$u_R$ is a solution also across $\{t=0\}$.

We have established the existence of a solution $u_R$ in 
${\mathcal  O}_R=B_R\cap\{s>t\}$ with $0\leq u_R\leq M$.  Considering
the odd reflection of $u_R$ with respect to the Simons cone 
${\mathcal C}$, 
$$
u_R(s,t)=-u_R(t,s),
$$
we obtain a solution in $B_R\setminus \{0\}$. 
Using the same cut-off argument as above, but choosing now
$1-\xi_{\varepsilon}$ to have support in the ball of radius
$\varepsilon$ around $0$, we conclude that $u_R$ is also solution 
around $0$, and hence in all of $B_R$. Here, the cut-off argument
also applies in dimension $2$.

We now wish to pass to the limit in $R$ and obtain a
solution in all of $\R^{2m}$.
For this, let $S>0$ and consider the family $\{u_R\}$, for ${R>S+2}$,
of solutions in $B_{S+2}$. 
Since $|u_R|\leq M$, interior elliptic estimates applied
in balls of radius~$2$ centered at points in $\overline{B}_S$
(as explained in the beginning of this section)
give a uniform $C^{2,\alpha}(\overline{B}_S)$ bound for
$u_R$ (uniform with respect to $R$). For later purposes,
using the argument in \eqref{rad1} for $u_R$, we have
\begin{equation}\label{grad1}
|\nabla u_R|\leq C \quad\text{ in } B_S, \qquad\text{for all }
R>S+2,
\end{equation}
for some constant $C$ independent of $S$ and $R$.
In addition, by the Arzela-Ascoli theorem, a
subsequence of $\{u_{R}\}$ converges in
$C^2(\overline{B}_S)$ to a solution in $B_S$.
Taking $S=1,2,3,\ldots$ and making a Cantor diagonal argument, we
obtain a sequence $u_{R_j}$ converging in $C^2_{\rm loc}(\R^{2m})$ 
to a solution $u\in C^2(\R^{2m})$.

By construction, we have that $u$ is a solution in $\R^{2m}$ depending
only on $s$ and $t$, odd with respect to the Simons cone 
${\mathcal C}$, with $|u|\leq M$ in $\R^{2m}$, and with $u\geq 0$ in $\{s>t\}$. 
Now, using that $f(M)=0$ and $u\not\equiv M$ (since
$u$ vanishes on ${\mathcal C}$), the strong maximum principle
gives that $u<M$ everywhere. As a consequence, we also have $u>-M$.

We claim that $u\not\equiv 0$ in $\R^{2m}$. Then, the
strong maximum principle leads to $u>0$ in $\{s>t\}$,
since $f(0)=0$ and $u\geq 0$ in $\{s>t\}$.
Thus, $u$ has all the properties of a saddle
solution as in Definition~\ref{def}.

To show that $u\not\equiv 0$, let $1<S<R-2$ and $w_R$ be defined as
$$w_R=\xi\min\left\{M, \frac{s-t}{\sqrt{2}}\right\} +(1-\xi)u_R,$$ 
where $\xi$ is a smooth function depending only on $r^2=s^2+t^2$ 
such that $\xi\equiv 1$ 
in $B_{S-1}$ and $\xi\equiv 0$ outside $B_S$. 
We have that $w_R\in \tilde{H}^1_0({\mathcal O}_R)$
satisfies 
\begin{equation}\label{ese}
\left\{
\begin{array}{rcll}
w_R & = & u_R & {\rm in \, } {\mathcal O}_R\setminus {\mathcal O}_S\\
w_R & = & \min\left\{M, \dfrac{s-t}{\sqrt{2}}\right\} & {\rm in \, }
{\mathcal O}_{S-1}.
\end{array}
\right.
\end{equation}
In addition, by \eqref{grad1}, we have
\begin{equation}\label{grad2}
|\nabla w_R| \leq C  \quad\text{ in }
{\mathcal O}_{S} 
\end{equation}
for some constant $C$ independent of $S$ and $R$.

Since $u_R$ minimizes the energy in $\tilde{H}^1_0({\mathcal O}_R)$, 
we have that
${\mathcal E}(u_R,{\mathcal O}_R)\leq {\mathcal E}(w_R,{\mathcal O}_R )$. 
Now, since $w_R = u_R$ in ${\mathcal O}_R\setminus {\mathcal O}_S$, 
we must have, for constants $C$ independent of $S$ and $R$,
\begin{eqnarray*}
&& \hspace{-3mm} \int_{{\mathcal O}_S} 
\left\{ \frac{1}{2}|\nabla u_R|^2+G(u_R)\right\}dx \\
&& \leq  \int_{{\mathcal O}_S} \left\{ \frac{1}{2}|\nabla
w_R|^2+G(w_R)\right\}dx\\
&& =   \int_{{\mathcal O}_{S-1}}
\left\{ \frac{1}{2}|\nabla w_R|^2+G(w_R)\right\} dx+ \int_{{\mathcal O}_S\setminus
{\mathcal O}_{S-1}} \left\{ \frac{1}{2}|\nabla
 w_R|^2+G(w_R)\right\}dx\\
&&  \leq  C\left| {\mathcal O}_{S-1}\cap\left\{\dfrac{s-t}{\sqrt{2}}<M\right\}
\right|+   
C|{\mathcal O}_S\setminus {\mathcal O}_{S-1}|\\
&&\leq C \int_0^{S-1} \left\{(t+\sqrt{2}M)^m-t^m\right\}t^{m-1}\; dt + 
C|B_S\setminus B_{S-1}|\\
&& \leq
 CS^{2m-1}.
\end{eqnarray*}
We have used the uniform gradient bound \eqref{grad2}, the equality
in ${\mathcal O}_{S-1}$ stated in \eqref{ese}, and $G(M)=0$. 
We have also used that $dx$ is equal to $c_ms^{m-1}t^{m-1}ds dt$
to bound the measure of the subset of ${\mathcal O}_{S-1}$,
and that $(t+\sqrt{2}M)^m-t^m\leq C t^{m-1}$ and $S^{2m}-(S-1)^{2m}
\leq CS^{2m-1}$ for $t$ and $S$ larger than $1$.
We now let $R=R_j\rightarrow\infty$ to obtain 
$$
\int_{{\mathcal O}_S} \left\{ \frac{1}{2}|\nabla u|^2+G(u)\right\}dx
\leq CS^{2m-1}
$$
for some constant $C$ independent of $S$. Note that this
bound, after odd reflection with respect to ${\mathcal C}$,
establishes the energy bound 
\begin{equation}\label{enerb}
{\mathcal E}(u,B_S)
\leq C S^{2m-1},
\end{equation}
which is estimate \eqref{est-ener} in the statement of the theorem.

Suppose that $u\equiv 0$. Then the energy bound \eqref{enerb}
would read
$$
c_mG(0)S^{2m}=G(0)|B_S|={\mathcal E}(0,B_S)\leq CS^{2m-1}.
$$ 
This is a contradiction for $S$ large, and thus $u\not\equiv 0$.

Finally, we establish the last statement of the theorem on 
stability under perturbations vanishing on the Simons cone.
We assume hypothesis \eqref{H3} on the concavity of $f$
in $(0,M)$. Since $f(0)=0$, concavity leads to $f'(w)\leq f(w)/w$ for all 
real numbers $w\in (0,M)$. Hence we have
$$
-\Delta u =f(u) \geq f'(u) u \qquad\text{in } {\mathcal O}.
$$
That is, $u$ is a positive supersolution for the linearized
operator $-\Delta - f'(u)$ at $u$ in all of ${\mathcal O}$.
By a simple argument (see the proof of 
Proposition 4.2 of \cite{AAC}),
it follows that the value of the quadratic form $Q_u(\xi)$ 
is nonnegative for all $\xi\in C^1$ with compact support in 
${\mathcal O}$ (and not necessarily depending
only on $s$ and $t$). By an approximation argument, the same
holds for all $\xi\in C^1$ with compact support in 
$\overline{\mathcal O}$ and vanishing on $\partial{\mathcal O}=
{\mathcal C}$. Finally, by odd symmetry with respect
to ${\mathcal C}$, the same is true
for all $C^1$ functions $\xi$ with compact support in $\R^{2m}$
and vanishing on~${\mathcal C}$.
\end{proof}

\section{Pointwise estimate for saddle solutions}

In this section we prove  Proposition $\ref{prop}$ using
an important estimate of Modica \cite{M1} and two elementary lemmas. 
In \cite{M1} Modica proved the following pointwise gradient bound 
for global solutions of semilinear elliptic equations.

\begin{theorem}{\bf (Modica \cite{M1})}\label{Mgrad}
Let $G\in C^2(\R)$ be a nonnegative
function and $u$ be a bounded solution of $\Delta
u-G'(u)=0$ in $\R^n$. Then,
\begin{equation}\label{modica}
\frac{|\nabla u|^2}{2}\leq G(u) \quad {\rm in}\,\R^n.
\end{equation}
In addition, if
$G(u(x_0))=0$ for some  $x_0\in\R^n$, then $u$ is constant.
\end{theorem}

In \cite{M1} this bound was proved under the hypothesis 
$u\in C^3(\R^n)$. The result as stated above, which applies to all
solutions ---recall that every weak solution is $C^{2,\alpha}(\R^n)$
since $G\in C^2(\R)$--- was established in \cite{CGS}.
 
Note that $1$-D solutions ---the functions $u_{b,c}$ defined in ($\ref{1d}$)--- 
make ($\ref{modica}$) to be an equality (see Lemma $\ref{lemma1D}$). 
In 1994, Caffarelli, Garofalo, and Seg{\`a}la \cite{CGS}
extended the previous result of Modica to a wider family of
equations which includes operators such as  the p-Laplacian and the
mean curvature operator for graphs. They also established that if
equality holds in $(\ref{modica})$ at some point of $\R^n$, then
$u$ must a 1-D solution.

The following are two auxiliary  lemmas towards Proposition
$\ref{prop}$.
The first one provides a formula for
the distance to the cone ${\mathcal C}$. The second one concerns
increasing solutions of ($\ref{eq}$) in $\R$.

\begin{lemma}\label{lemmadist}For every point $x\in\R^{2m}$, 
the distance from $x$ to
the Simons cone ${\mathcal C}=\{s=t\}$ is given by
$${\rm dist}(x,{\mathcal C})=\frac{|s-t|}{\sqrt{2}}.$$
\end{lemma}

This formula can be found, and also proven rigorously,
using the method of Lagrange multipliers. Next we give an
alternative simple proof of it.

\begin{proof}[Proof of Lemma $\ref{lemmadist}$]
Let $x=(x^1,x^2)\in\R^{2m}\setminus{\mathcal C}$ and
$x_0=(x_0^1,x_0^2)\in{\mathcal C}$. Let $s=|x^1|,t=|x^2|$, and
$s_0=t_0=|x_0^1|=|x_0^2|$. We have
\begin{eqnarray*}
|x-x_0|^2=|x^1-x_0^1|^2+|x^2-x_0^2|^2& = 
& s^2+t^2+2s_0^2-2x^1\cdot x_0^1-2x^2\cdot x_0^2\\
 & \geq & s^2+t^2+2s_0^2-2(s+t)s_0\\
 & = & \frac{(s-t)^2}{2} + \frac{1}{2}{\big (}(s+t)-2s_0 {\big
 )}^2\\
 & \geq & {\Big (}\frac{s-t}{\sqrt{2}}{\Big )}^2.
 \end{eqnarray*}

Next,  given $x\in \R^{2m}$ we show that $x_0\in{\mathcal C}$ can be
 chosen so that the two inequalities above are in fact equalities.
In case that $s>0$ and $t>0$, choose $x_0=(\alpha x^1, \beta x^2)=(\alpha
x_1,\dots,\alpha x_m,\beta x_{m+1},\dots, \beta
 x_{2m})$, where $\alpha$ and $\beta$ are given by
 $\alpha s=\beta t=(s+t)/2$. If either $s$ or $t$ are zero,
say $s>0$ and $t=0$, choose $x_0=( x^1/2, x^1/2)$.
\end{proof}

The proof of the following lemma, which follows from
integrating the ODE $\ddot{u}-G'(u)=0$, can be found in \cite{AC}
---see also a sketch of the proof below, after the statement.

\begin{lemma}\label{lemma1D} Let $G\in C^2(\R)$.
There exists a bounded function $u_0\in C^2(\R)$ satisfying
$$\ddot{u}_0-G'(u_0)=0 \quad \hbox{and} \quad \dot{u}_0>0 
\quad \hbox{ in } \R
$$
if and only if there exist two real numbers $m_1<m_2$ for which
$G$ satisfies
\begin{equation}\label{G1}
G'(m_1)=G'(m_2)=0 \qquad\text{ and }
\end{equation}
\begin{equation}\label{G2}
G>G(m_1)=G(m_2) \quad \hbox{in } (m_1,m_2).
\end{equation}
In such case we have $m_1=\lim_{\tau\rightarrow -\infty} u_0(\tau)$ and 
$m_2=\lim_{\tau\rightarrow +\infty} u_0(\tau)$. Moreover, the solution
$u_0=u_0(\tau)$ is unique up to translations of the independent 
variable $\tau$. 

Adding a
constant to $G$, assume that
\begin{equation}\label{Gzero}
G(m_1)=G(m_2)=0.
\end{equation}
Then, we have that
\begin{equation}\label{hamilt}
\frac{\dot{u}_0^2}{2}= G(u_0) \qquad\text{in }\R .
\end{equation}
If in addition 
\begin{equation}\label{G3}
G''(m_1)\neq 0 \quad\text{and}\quad G''(m_2)\neq 0,
\end{equation}
then
\begin{equation}\label{cresc1D}
0<\dot{u}_0(\tau)\leq C e^{-c|\tau|} \quad \text{in }\R
\end{equation}
for some positive constants $C$ and $c$, and
\begin{equation}\label{int1D}
\int_{-\infty}^{+\infty} \left\{ \frac{1}{2}\dot{u}_0(\tau)^2 +
G(u_0(\tau)) \right\} d\tau <+\infty .
\end{equation} 
\end{lemma}

Given G satisfying ($\ref{G1}$), ($\ref{G2}$), and \eqref{Gzero},
to construct $u_0$ 
we simply choose any $m_0\in(m_1,m_2)$ and define
$$\phi(\sigma)=\int_{m_0}^{\sigma}\frac{dw}{\sqrt{2(G(w))}} \quad\text{ for }
\sigma\in (m_1,m_2).$$
Then let $u_0:=\phi^{-1}$ be the inverse function of $\phi$.
This formula is found multiplying $\ddot{u}-G'(u)=0$ by $\dot{u}$
and integrating the equation ---which also gives the necessity
of conditions ($\ref{G1}$) and ($\ref{G2}$) for existence.
The above definition of $u_0$ leads automatically to \eqref{hamilt}.

Under the hypothesis
$G''(m_i)\neq 0$, $G$ behaves like a quadratic function
near each $m_i$. Using the expression above for $\phi$, this gives
that $\phi$ blows-up logarithmically at $m_i$, and thus 
its inverse function $u_0$
attains its limits $m_i$ at $\pm\infty$ exponentially. 
From this and identity \eqref{hamilt}, the exponential decay
\eqref{cresc1D} for $\dot{u}_0$ follows, as well as \eqref{int1D}.

Next we prove our pointwise bound.

\begin{proof}[Proof of Proposition $\ref{prop}$]  
Let $u$ be a bounded solution of $-\Delta u=f(u)$ in $\R^{2m}$ that
vanishes on the Simons cone ${\mathcal C}=\{s=t\}$. We wish to show
that
$$|u(x)|\leq \left|u_0\left(\frac{s-t}{\sqrt{2}}\right)\right|
\qquad\text{in }\R^{2m}.$$

First we prove that $|u|<M$. Arguing by
contradiction, assume that $|u|\geq M$ somewhere. Since $u(0)=0$,  
there exists a
point $x_0$ such that $u(x_0)=\pm M$. Then, by Modica gradient bound
($\ref{modica}$) we have that
$|\nabla u(x_0)|^2\leq 2G(u(x_0))=2G(\pm M)=0$.
Therefore $G(u(x_0))=0$ and, by the second part of
Theorem~$\ref{Mgrad}$, $u$ is constant. Since $u=0$ 
on the Simons cone, we must have $u\equiv 0$. This contradicts the assumption
$|u|\geq M>0$ somewhere.

Next, since $|u|<M$, we may write
$$u(x)=u_0(v(x))$$ for some function $v:\R^{2m}\rightarrow\R$, 
where $u_0$ is the $1$-D
solution whose existence is given by Lemma $\ref{lemma1D}$, with
$m_1=-M$ and $m_2=M$, and such that
$u_0(0)=0$. Now, Modica
estimate ($\ref{modica}$) written in terms of $v$ becomes
$$
\frac{1}{2}\dot{u}_0^2(v)|\nabla v|^2\leq
G(u_0(v))\quad\text{ in }\R^{2m}.
$$ 
Since $\dot{u}_0^2/2\equiv G(u_0)$ by \eqref{hamilt}, the
expression above leads to
$$|\nabla v|\leq 1 \quad\text{ in }\R^{2m}.$$

Finally, since $u=0$ on ${\mathcal C}$, we also have $v=0$ on 
${\mathcal C}$. Given $x\in\R^{2m}$, let
 $x_0\in{\mathcal C}$ be such that $|x-x_0|={\rm dist}(x,{\mathcal
  C})$. Then,
$$|v(x)|=|v(x)-v(x_0)|\leq \vectornorm{\nabla v}_{L^{\infty}}|x-x_0|
\leq |x-x_0|={\rm dist}(x, {\mathcal C}).$$
By Lemma $\ref{lemmadist}$, using that $u_0$ is odd since $f$ is
odd and that $u_0$ is increasing, we conclude 
\begin{equation*}
|u(x)|=|u_0(v(x))|=u_0(|v(x)|)\leq u_0({\rm dist}(x,{\mathcal C}))=
\left|u_0{\Big (}\frac{s-t}{\sqrt{2}}{\Big )}\right|,
\end{equation*}
which is the desired bound.

Finally, we prove the last statement of the proposition, that is,
the fact that $u_0((s-t)/\sqrt{2})$
is a supersolution of $-\Delta u=f(u)$ in ${\mathcal O}=\{s>t\}$.
First, by direct computation using equation \eqref{eqst} 
in $(s,t)$ variables for $t>0$ gives that the function is a 
supersolution in $\{s>t>0\}$. In dimension $2m\geq 4$ there is nothing
else to be checked, by a capacity (or cut-off) argument used
as in \eqref{capa}. Instead, in dimension 2,  $u_0((s-t)/\sqrt{2})$ is a
supersolution in ${\mathcal O}$ since the outer flux 
$-\partial_t u_0((s-t)/\sqrt{2})\mid_{t=0}=\dot{u}_0(s/\sqrt{2})/
\sqrt{2}>0$ is positive.
\end{proof}

\section{Instability in dimension $n=4$}

In this section we prove the instability result of Theorem $\ref{uns}$.
For this, we establish that
the function $u_0((s-t)/\sqrt{2})$ is unstable in dimension $4$
in the sense that the second variation of energy $Q$ at
$u_0((s-t)/\sqrt{2})$ is negative for some test
function $\xi$ depending only on $s$ and $t$. 
Our proof also gives its instability outside of
every compact set.
Even that $u_0((s-t)/\sqrt{2})$
is not a solution of the equation, we define
the quadratic form
\begin{equation}\label{Qu0}
Q_{u_0}(\xi):=\int_{\R^{2m}}\left\{|\nabla\xi(x)|^2-
f'\left( u_0( (s-t)/\sqrt{2}) \right)
\xi^2(x)\right\}dx ,
\end{equation}
where there is some abuse of notation in writting $Q_{u_0}$ since
by $u_0$ we really mean $u_0((s-t)/\sqrt{2})$.

The key point of the proof is that $Q_{u_0}$ not being nonnegative 
definite leads to the same property for $Q_{u}$, where $u$ is
any bounded solution that vanishes on the Simons cone. This fact will follow
from our main pointwise bound of Proposition~\ref{prop}.

For the proof it is useful to consider the variables
$$
\left\{ \begin{array}{rcl}
y&=&{\ds \frac{s+t}{\sqrt{2}}}\\
z&=&{\ds \frac{s-t}{\sqrt{2}}} ,
\end{array}
\right.
$$
which satisfy $-y\leq z\leq y$.

Recall that a  bounded solution $u$  of $-\Delta u=f(u)$ in $\R^{2m}$
is stable provided
$$Q_u(\xi)=\int_{\R^{2m}}\left\{|\nabla\xi|^2-f'(u)\xi^2\right\}dx 
\geq 0\qquad
\text{ for all } \xi\in C_c^{\infty}(\R^{2m}).$$
If $v$ is a function depending only on $s$ and $t$, the quadratic form
$Q_v(\xi)$ acting on perturbations of the form $\xi=\xi(s,t)$ becomes
$$c_m Q_v(\xi)= \int_{\{s> 0, t> 0\}} s^{m-1}t^{m-1}\left\{\xi_s^2+\xi_t^2-
f'(v)\xi^2\right\}dsdt ,$$
where $c_m>0$ is a constant depending only on $m$.
We can further change to variables $(y,z)$ and obtain, 
for a different constant $c_m>0$, 
\begin{equation}\label{five2}
c_m Q_v(\xi)=\int_{\{-y<z<y\}}(y^2-z^2)^{m-1}\left\{\xi_y^2+\xi_z^2-
f'(v)\xi^2\right\}dydz .
\end{equation}

Given the definition of the variables $y$ and $z$,
a function $\xi=\xi(y,z)$ has compact support in $\R^{2m}$ if and only if 
$\xi(y,z)$ vanishes for $y$ large enough.

\begin{proof}[Proof of Theorem $\ref{uns}$] 
Let $u$ be a bounded
solution of $-\Delta u=f(u)$ in $\R^{2m}$ vanishing on the
Simons cone ${\mathcal C}=\{s=t\}$.
By Proposition~\ref{prop}, we know that 
$$
|u(x)|\leq |u_0(z)|
\qquad\text{ in all of } \R^{2m}.
$$ 
This leads to
$f'(|u(x)|)\geq f'(|u_0(z)|)$ for all $x\in\R^{2m}$, since we
assume $f$ to be concave in $(0,M)$. Now, since $f'$ is even,
we deduce that
$$
f'(u(x))\geq f'(u_0(z))\qquad\text{ for all  } x\in\R^{2m}.
$$ 
Therefore, since $Q_{u_0}$ is defined by \eqref{Qu0} and
$(s-t)/\sqrt{2}=z$, we conclude
\begin{equation}\label{ineqQ}
Q_u(\xi)\leq Q_{u_0}(\xi)\qquad\text{ for all } 
\xi\in C^{\infty}_c(\R^{2m}).
\end{equation}

It follows that, in order to prove that $u$ is unstable, it
suffices to find a smooth function $\xi$ with compact support 
in $\R^{2m}$ for which $Q_{u_0}(\xi)< 0$. This is an easier task since $u_0(z)$
is explicit. Note also that, by an
approximation argument, it suffices to find a Lipschitz function 
$\xi$, not necessarily smooth,  with compact support in $\R^{2m}$ and 
for which $Q_{u_0}(\xi)< 0$.

Expression \eqref{five2} with $v(x)=u_0(z)$ reads
\begin{equation*}
c_m Q_{u_0}(\xi)=\int_{\{-y<z<y\}}(y^2-z^2)^{m-1}\left\{\xi_y^2
+\xi_z^2-f'(u_0(z))\xi^2\right\}dydz
\end{equation*}
for all $\xi=\xi(y,z)$ with compact support in $\R^{2m}$. 
We take now $\xi$ of separate variables, that is, of
the form $$\xi(y,z)=\phi(y)\psi(z).$$ 
For $\xi$ to have 
compact support in $\R^{2m}$ it suffices that $\phi$ has compact support 
in $y\in (0,+\infty)$ (with no requirement on the support of $\psi$). 
Note also that $\xi$ is a Lipschitz function
of $x\in\R^{2m}$ if $\phi$ and $\psi$ are Lipschitz. 
However, even if $\phi$ and $\psi$ are smooth, $\xi$ will not be in
general better than Lipschitz ---to be smooth it would be necessary
that the normal derivatives of $\xi$ vanish
at $s=0$ and $t=0$ (i.e., at $z=\pm y$).

Since $\xi_y^2+\xi_z^2=\phi_y^2\psi^2+
\phi^2\psi_z^2$, we have
\begin{equation}\label{eqsep}
c_m Q_{u_0}(\xi)=\int_{\{-y<z<y\}} (y^2-z^2)^{m-1}\{\phi_y^2\psi^2+
\phi^2\psi_z^2-f'(u_0(z))\phi^2\psi^2\}dydz .
\end{equation}

Choose $$\psi(z)=\dot{u}_0(z).$$ We now let
$2m=4$ and thus $m-1=1$. In the following computations,
we first integrate by parts the term $\{(y^2-z^2)\phi^2\psi_z\}\psi_z$
with respect to $z$ (note that here we obtain no boundary terms), 
and later we write the term
$2z\phi^2\psi\psi_z$ as $\phi^2z(\psi^2)_z$ and we integrate it 
by parts with respect to $z$. Thus,
\begin{eqnarray*}
& & \hspace{-10mm} 
c_m Q_{u_0}(\xi) =  \int_{\{-y<z<y\}}
(y^2-z^2)\{\phi_y^2\psi^2+
\phi^2\psi_z^2-f'(u_0(z))\phi^2\psi^2\}dydz \\
& = & \int_{\{-y<z<y\}} (y^2-z^2)\phi_y^2\psi^2dydz+ 
\int_{\{-y<z<y\}}
2z\phi^2\psi\psi_zdydz\\
& & \qquad\qquad -\int_{\{-y<z<y\}}
(y^2-z^2)\phi^2\psi\{\psi_{zz}+f'(u_0(z))\psi\}dydz\\
& = & \int_{\{-y<z<y\}} (y^2-z^2)\phi_y^2\psi^2dydz - 
\int_{\{-y<z<y\}}
\phi^2\psi^2dydz\\
&  & + \int_{0}^{+\infty} \phi^2(y)[z\psi^2(z)]_{-y}^{y}dy\\
& = & \int_{\{-y<z<y\}} (y^2-z^2)\phi_y^2\psi^2dydz - 
\int_{\{-y<z<y\}} \phi^2\psi^2dydz \\
&  & + \int_{0}^{+\infty} \phi^2(y) 2y \psi^2(y)dy\\
& \leq & \int_{\{-y<z<y\}} y^2\phi_y^2\psi^2dydz - 
\int_{\{-y<z<y\}} \phi^2\psi^2dydz +
\int_{0}^{+\infty} \phi^2(y) 2y\psi^2(y) dy,
\end{eqnarray*}
where we have used that $\psi=\dot{u}_0$ is an even function and
a solution to the
linearized $1$-D problem ${\psi}_{zz}+f'(u_0(z))\psi=0$.

For $a>1$, a constant that we will make tend to infinity, let 
$\eta=\eta(\rho)$ be a Lipschitz function of $\rho:=y/a$ with compact support 
$[\rho_1,\rho_2]\subset (0,+\infty)$. Let us denote by
$$
\phi(y)=\phi_a(y)=\eta(y/a) \qquad\text{and}\qquad
\xi_a(y,z)=\phi_a(y)\dot{u}_0(z)=\eta(y/a)\dot{u}_0(z)
$$
the functions named $\phi$ and $\xi$ above.
In the last bound for $Q_{u_0}$, we make the change 
$y=a\rho$, $dy=ad\rho$, and we use that $\psi=\dot{u}_0$ is decreasing
in~$(0,+\infty)$. We obtain
\begin{eqnarray*}
& & \hspace{-10mm} c_m Q_{u_0}(\xi_a)\leq \\
& \leq & \int_{\{-y<z<y\}} y^2\phi_y^2\psi^2dydz - 
\int_{\{-y<z<y\}}
\phi^2\psi^2dydz + \int_{0}^{+\infty} \phi^2(y)2y\psi^2(y)dy\\
& \leq &  \int_{\{-a\rho<z<a\rho\}} a^3\rho^2\frac{\eta_{\rho}^2}{a^2}\psi^2d\rho
dz - \int_{\{-a\rho<z<a\rho\}}
a\eta^2\psi^2d\rho dz\\
& & \qquad\qquad  + \int_{\rho_1}^{\rho_2} a\eta^2(\rho)2a\rho\psi^2(a\rho)d\rho\\
& = & a\left \{  \int_{0}^{+\infty}
\rho^2\left\{ \eta_{\rho}^2-\frac{\eta^2}{\rho^2}\right\} \left\{\int_{-a\rho}^{a\rho}
\dot{u_0}^2dz \right\} d\rho  +
2a\rho_2\dot{u}_0^2(a\rho_1)\int_{0}^{+\infty}\eta^2d\rho\right\}.
\end{eqnarray*}
Dividing by $a$, this leads to
\begin{equation*}
\frac{c_mQ_{u_0} (\xi_a)}{a}\leq \int_{0}^{+\infty}
\rho^2\left\{\eta_{\rho}^2-\frac{\eta^2}{\rho^2}\right\}
\left\{\int_{-a\rho}^{a\rho} \dot{u}_0^2dz \right\}d\rho  +
2a\rho_2\dot{u}_0^2(a\rho_1) \int_{0}^{+\infty}\eta^2d\rho .
\end{equation*}

Since $f$ is a concave function in $(0,M)$
with $f(0)=f(M)=0$, we have $f\geq 0$ in $(0,M)$. In addition,
since $f\not\equiv 0$ in $(0,M)$ by \eqref{H2}, we must have
$f>0$ in $(0,M)$ by concavity. Now, $f$ being concave and positive in $(0,M)$
and with $f(M)=0$, we deduce that $f'(M)<0$. Hence, hypothesis
\eqref{G3} in Lemma~\ref{lemma1D}, $G''(\pm M)>0$ is satisfied. Thus by
($\ref{cresc1D}$), we conclude that
$$\lim_{a\rightarrow +\infty}a\dot{u}_0^2(a\rho_1)=0.$$
Therefore, letting $a\to +\infty$ in the last bound for $Q_{u_0}$ we obtain
\begin{equation}\label{concl}
\limsup_{a\rightarrow +\infty} \frac{c_mQ_{u_0} (\xi_a)}{a}\leq
\left\{\int_{-\infty}^{+\infty} \dot{u}_0^2(z)dz \right\}
\int_{0}^{+\infty}
 \rho^2\left\{\eta_{\rho}^2-\frac{\eta^2}{\rho^2}\right\}d\rho.  
\end{equation}
By $(\ref{int1D})$ in Lemma $\ref{lemma1D}$,
$\dot{u}_0^2$ is integrable in $(-\infty,\infty)$. 
Thus, by \eqref{ineqQ} and the comments after it,
the proof of the instability of the solution $u$ will be finished
if there exists a Lipschitz function $\eta=\eta(\rho)$ 
with compact support in $(0,+\infty)$ for which the second
integral in \eqref{concl} is negative.

Arguing by contradiction, assume the contrary and therefore that 
\begin{equation}\label{rho}
\int_{0}^{+\infty} \rho^2\frac{\eta^2}{\rho^2}d\rho\leq
\int_{0}^{+\infty}\rho^2\eta_{\rho}^2d\rho, 
\end{equation}
for every Lipschitz function $\eta=\eta(\rho)$ 
with compact support in $(0,+\infty)$. The requirement that
$\eta$ vanishes in a neighborhood of $0$ can be removed
by simply cutting-off $\eta$ in $(0,\varepsilon)$ and
letting $\varepsilon\to 0$.
The integrals in \eqref{rho} can be seen as integrals in
$\R^3$ of radial functions, that is, functions of the radius $\rho$.  
Hardy's inequality in $\R^3$ states that 
$$
\frac{(3-2)^2}{4}\int_{\R^3} \frac{\eta^2}{|x|^2}dx \leq 
\int_{\R^3}|\nabla
 \eta|^2dx 
$$ 
holds for every Lipschitz function $\eta$ with compact support in
$\R^3$, and that the  constant 
$(3-2)^2/4=1/4$ is the best possible even when the inequality is
considered only among radial functions. Hence, since
$$
1>\frac{1}{4}=\frac{(3-2)^2}{4},
$$ 
($\ref{rho}$) leads to a contradiction, and this finishes the proof
of instability.

The following is a direct way (without using Hardy's inequality)  
to see that the second integral in \eqref{concl} is negative 
for some Lipschitz function $\eta$ with 
compact support in $(0,\infty)$. 
For $\alpha >0$ and $0<2\rho_1<1<\rho_2$, let
\begin{equation*}
\eta(\rho)= \left\{
\begin{array}{ll}
(1-\rho_2^{-\alpha})\rho_1^{-1}(\rho-\rho_1) & \text{ for } \rho_1\leq \rho \leq 2\rho_1\\
1-\rho_2^{-\alpha} & \text{ for }2\rho_1\leq \rho \leq 1\\
\rho^{-\alpha}-\rho_2^{-\alpha} & \text{ for } 1\leq \rho \leq \rho_2\\
0 & \text{ otherwise},
\end{array}
\right.
\end{equation*}
a Lipschitz function with compact support $[\rho_1,\rho_2]$.
We simply compute the second integral in \eqref{concl} and find
\begin{eqnarray*}
& & \hspace{-10mm}  \int_{0}^{+\infty}
 \rho^2\left\{\eta_{\rho}^2-\frac{\eta^2}{\rho^2}\right\}d\rho \leq \\
& \leq &
 \int_{\rho_1}^{2\rho_1}\rho^2 (1-\rho_2^{-\alpha})^2 \rho_1^{-2} d\rho + 
\int_1^{+\infty} \alpha^2\rho^{-2\alpha}d\rho
- \int_1^{\rho_2} (\rho^{-\alpha}-\rho_2^{-\alpha})^2 d\rho.
\end{eqnarray*}
Choosing $1/2<\alpha<1$, as $\rho_2\to\infty$ the difference of the
last two integrals converges to a negative number,
since $\alpha^2<1$. Since the first of the three last integrals is
bounded by $3\rho_1$, we conclude that the above expression is
negative by chossing $\rho_2$ large enough and then $\rho_1$
small enough.

The previous proof of instability also leads to the instability
outside of every compact set ---and thus to the infinite Morse
index property of $u$. Indeed, choosing $\rho_1$ and
$\rho_2$ (and thus $\eta$) as above, we 
consider the corresponding function $\xi_a$ for $a>1$. Now,
\eqref{ineqQ} and \eqref{concl} lead to 
$Q_u (\xi_a)\leq Q_{u_0}(\xi_a) <0$ for $a$ large enough. 
Thus, the Lipschitz function $\xi_{a}$ makes $Q_u$
negative for $a$ large, and has compact support contained in 
$\{a\rho_1 \leq (s+t)/\sqrt{2}\leq a\rho_2\}$. By approximation,
the same is true for a function $\xi$ of class $C^1$, not
only Lipschitz.
Hence, given any compact set $K$ of $\R^{4}$, by taking $a$ 
large enough we conclude that $u$ is unstable outside $K$,
as stated in the theorem.

{From} the instability outside every compact set, it follows
that $u$ has infinite Morse index in the sense of 
Definition~\ref{def1.1}(c). Indeed, let $X_k$ be a subspace
of $C^1_c(\R^{4})$ of dimension $k$, generated by functions
$\xi_1,\ldots,\xi_k$, and with $Q_u(\xi)<0$ for all 
$\xi\in X_k\setminus\{0\}$.
Let $K$ be a compact set containing the support of all
the functions $\xi_1,\ldots,\xi_k$. Since $u$ is unstable
outside $K$, there is a $C^1$ function $\xi_{k+1}$ with 
compact support
in $\R^{4}\setminus K$ for which $Q_u(\xi_{k+1})<0$.
Since $\xi_{k+1}$ has disjoint support with each of the functions
$\xi_1,\ldots,\xi_k$, it follows that $\xi_1,\ldots,\xi_k,\xi_{k+1}$
are linearly independent and that $Q_u(a_1\xi_1+\cdots+a_{k+1}\xi_{k+1})
=Q_u(a_1\xi_1+\cdots+a_{k}\xi_{k})+Q_u(a_{k+1}\xi_{k+1})<0$ for every
nonzero linear combination $a_1\xi_1+\cdots+a_{k+1}\xi_{k+1}$ of them.
We conclude that $u$ has infinite Morse index.
\end{proof}

\section{Asymptotic stability of $u_0(z)$ in  dimensions $2m\geq 6$}

In this section we carry out, in all dimensions $n=2m\geq 4$, the
asymptotic analysis done for $n=4$ in the proof of Theorem $\ref{uns}$. 
We will see that the argument does not lead to the instability of saddle solutions
in dimensions $n=2m\geq 6$. Indeed, we will show that $u_0(z)=u_0((s-t)/\sqrt{2})$ 
is, in every dimension $n=2m\geq 6$ and only in some weak sense, 
asymptotically stable with respect to 
perturbations $\xi(y,z)$ with separate variables. 
Although this applies in dimension $6$, note that $Q_u\leq Q_{u_0}$ is only an
inequality and thus, $u_0$ may be asymptotically stable (or even could be stable)
and at the same time the solution $u$ be unstable. Indeed, a more recent result of us
\cite{CT} establishes that saddle
solutions in dimension $6$ are unstable.

Recall that by \eqref{eqsep}, the second variation of energy at $u_0(z)$
applied to test functions $\xi(y,z)=\phi(y)\psi(z)$ has the form 
$$
c_m Q_{u_0}(\xi)=\int_{\{-y<z<y\}} (y^2-z^2)^{m-1}\{\phi_y^2\psi^2+
\phi^2\psi_z^2-f'(u_0(z))\phi^2\psi^2\}dydz.
$$

We choose, as in Theorem $\ref{uns}$, $\psi(z)=\dot{u}_0(z)$ 
---which is a solution of the linearized problem in $z$ and thus 
it should be the most unstable perturbation in the $z$ variable. 
We proceed as in
the proof of Theorem~\ref{uns}, integrating by parts the term 
$\{(y^2-z^2)^{m-1}\phi^2\psi_z\}\psi_z$
with respect to $z$, and later re-writting the term
$(y^2-z^2)^{m-2}2z\phi^2\psi\psi_z$ as 
$(y^2-z^2)^{m-2}\phi^2z(\psi^2)_z$, and integrating it 
by parts with respect to $z$. Now there are no boundary terms when we
integrate by parts since $m-2\geq 1$. We obtain
\begin{eqnarray*}
c_m Q_{u_0}(\xi) & = & \int_{\{-y<z<y\}}
(y^2-z^2)^{m-1}\{\phi_y^2\psi^2+
\phi^2\psi_z^2-f'(u_0(z))\phi^2\psi^2\}dydz \\
 & = &  \int_{\{-y<z<y\}} (y^2-z^2)^{m-1}\{\phi_y^2\psi^2-
\phi^2\psi{\big (}\psi_{zz}+f'(u_0(z))\psi{\big )}\}dydz\\
& & \qquad\qquad + \int_{\{-y<z<y\}} 
(m-1)(y^2-z^2)^{m-2}2z\phi^2\psi\psi_zdydz 
\\
& = &  \int_{\{-y<z<y\}} (y^2-z^2)^{m-1}\psi^2\left\{\phi_y^2 -
\frac{m-1}{y^2-z^2}\phi^2\right\}dydz\\
& & \qquad\qquad + \int_{\{-y<z<y\}}
(m-1)(m-2)(y^2-z^2)^{m-3}2z^2\phi^2\psi^2dydz ,
\end{eqnarray*}
where we have used that $\psi=\dot{u}_0$ is a solution to the
linearized $1$-D problem $\psi_{zz}+f'(u_0(z))\psi =0$.

As before, let $a>1$ and $\phi(y)=\eta(y/a)$, where $\eta=\eta(\rho)$ 
is a function
of $\rho=y/a$ with compact support in $(0,+\infty)$. Let
$$
\xi_a(y,z)=\xi(y,z)=\eta(y/a)\dot{u}_0(z).
$$
Since $y=a\rho$,
$dy=ad\rho$, we have that
\begin{eqnarray*}
& & \hspace{-6mm} c_m Q_{u_0}(\xi_a) = \\
& &  = \int_{\{-a\rho\leq z \leq a\rho\}}
a^{2m-3}\rho^{2m-2}\left(1-\frac{z^2}{a^2\rho^2}\right)^{m-1}\dot{u}_0^2
\left\{\eta_y^2
-\frac{m-1}{\rho^2(1-\frac{z^2}{a^2\rho^2})}\eta^2\right\}d\rho dz \\
& & \quad + \int_{\{-a\rho\leq z \leq a\rho\}}
2(m-1)(m-2)a^{2m-5}\rho^{2m-6}\eta^2\dot{u}_0^2z^2\left(1-
\frac{z^2}{a^2\rho^2}\right)^{m-3}d\rho dz.
\end{eqnarray*}

Dividing by $a^{2m-3}$ we deduce
\begin{eqnarray*}
&&\frac{c_mQ_{u_0}(\xi_a)}{a^{2m-3}} =  \\
&&\qquad =\int_{\{-a\rho\leq z\leq a\rho\}}
\rho^{2m-2}\left(1-\frac{z^2}{a^2\rho^2}\right)^{m-1}
\dot{u}_0^2\left\{\eta_y^2-
\frac{m-1}{\rho^2(1-\frac{z^2}{a^2\rho^2})}\eta^2\right\}d\rho dz \\
&&\qquad \quad +\int_{\{-a\rho\leq z \leq a\rho\}}
\frac{2(m-1)(m-2)}{a^{2}}\rho^{2m-6}\eta^2\dot{u}_0^2z^2\left(1-
\frac{z^2}{a^2\rho^2}\right)^{m-3}d\rho dz.
\end{eqnarray*}
Since $0\leq 1-z^2/(a^2\rho^2)\leq 1$, $m\geq 3$, and $z^2\dot{u}_0^2(z)$ is
integrable in $(-\infty,+\infty)$, as $a\to\infty$ the second integral 
tends to zero
and we are left with the expression
$$ 
\limsup_{a\to +\infty} \frac{c_mQ_{u_0}(\xi_a)}{a^{2m-3}} =
\left\{\int_{-\infty}^{+\infty} \dot{u}_0^2 dz\right\}
\int_{0}^{+\infty}
\rho^{2m-2}\left\{\eta_y^2-\frac{m-1}{\rho^2}\eta^2\right\}d\rho .
$$ 
According to Hardy's inequality for radial functions in $\R^{2m-1}$,
the last integral in $\rho$ 
is nonnegative for all $\eta=\eta(\rho)$ with compact support if and only if
$$
m-1\leq \frac{(2m-3)^2}{4}.
$$
When $2m\geq 6$ this inequality is true (it is even strict) and we conclude
some kind of asymptotic stability for $u_0(z)$. 

\vspace{1em}

{\sc Acknowledgment.} The first author thanks Regis Monneau for
interesting discussions and ideas on the subject of this paper.

\end{document}